\documentclass[12pt,a4paper]{amsart}
\usepackage{geometry}
\usepackage{amsmath,amssymb,amsfonts,amscd,amsthm,wasysym,color,enumitem,indentfirst,graphicx,booktabs,mathrsfs}
\usepackage[bookmarksnumbered,colorlinks,linktocpage,plainpages]{hyperref}
\hypersetup{pdfstartview={FitH}}
\usepackage{mathtools}
\usepackage{extarrows}

\numberwithin{equation}{section}
\newtheorem{theorem}{Theorem}[section]

\newtheorem{lemma}[theorem]{Lemma}
\newtheorem{proposition}[theorem]{Proposition}

\theoremstyle{definition}
\newtheorem{definition}[theorem]{Definition}

\newtheorem{example}[theorem]{Example}

\newtheorem{remark}[theorem]{Remark}

\newcommand{\LJ}{\Lambda^2_-}

\usepackage{tikz}

\allowdisplaybreaks[4]

\makeatletter
\@namedef{subjclassname@2020}{\textup{2020} Mathematics Subject Classification}
\makeatother

\date{}

\begin{document}
\title[$L_0,L_1$ and $2$-Cauchy-Fueter equation]{ A Pair of Multiplication-Type Operators in Quaternionic Analysis and the 2-Cauchy-Fueter Equation}
\date{\today}

\author{Yong Li}
\address[Yong Li]{School of Mathematics and Statistics, Anhui Normal University, Wuhu 241002, Anhui, People's Republic of China}
\email{leeey@ahnu.edu.cn}

\author{Yuchen Zhang*}
\address[Yuchen Zhang]{Institute of Mathematics, Academy of Mathematics and Systems Sciences, Chinese Academy of Sciences, Beijing 100190, China}
\email[Y. Zhang]{yuchen95@amss.ac.cn}

\thanks{
$\ast$ Corresponding author.
\\
This work is supported by the National Natural Science Foundation of China (Grant No. 12501100), the Anhui Provincial Natural Science Foundation (Grant No. 2508085QA031) and
the University Natural Science Research Project of Anhui Province (Grant No. 2022AH050175).}


\begin{abstract}

In this paper, we introduce a pair of multiplication-like operators, $L_0$ and $L_1$, which derive
$k$-regular functions from $(k+1)$-regular functions. The investigation of the inverse problem naturally leads to a deeper study of the 2-Cauchy-Fueter equation. In doing so, we provide a new acyclic resolution for the sheaf of $2$-regular functions $\mathcal{R}^{(2)}$. Furthermore, a complete topological characterization for the solvability of the $2$-Cauchy–Fueter equation is established. Specifically, we prove that the $2$-Cauchy–Fueter equation
$$\mathscr{D}^{(2)}f=g$$
 is solvable  for any $g$ satisfying  $\mathscr{D}_1^{(2)}g=0$ on a domain $\Omega\subset \mathbb{R}^4$ if and only if $H^3(\Omega, \mathbb{R}) = 0$, or equivalently, if and only if every real-valued harmonic function on $\Omega$ can be represented as the real part of a quaternionic regular function.
\end{abstract}
\maketitle
\tableofcontents
\section{Introduction}

The quaternion algebra $\mathbb{H}$, with basis ${1, \mathbf{i}, \mathbf{j}, \mathbf{k}}$, was introduced by Hamilton \cite{Hamilton66}. Quaternionic analysis—a generalization of complex analysis—encompasses a variety of distinct conceptual frameworks.

The theory of quaternionic regular functions represents the most natural generalization of holomorphic function theory in complex analysis, originating from the seminal work of Fueter \cite{Fueter35}. These functions serves as the quaternionic analogue of holomorphic functions, defined as those annihilated by the Cauchy–Fueter operator
\begin{equation*}
D = \frac{\partial}{\partial x_0}+\mathbf{i}\frac{\partial}{\partial x_1}+\mathbf{j}\frac{\partial}{\partial x_2}+\mathbf{k}\frac{\partial}{\partial x_3}.
\end{equation*}
Fueter established numerous properties of quaternionic regular functions that parallel those of holomorphic functions—including Laurent expansions and a Cauchy integral formula \cite{Fueter35,Fueter36}. And it have extensive applications in theoretical physics \cite{Adler86,Lanczos29}. Over time, this function theory has matured considerably \cite{Sudbery79, CSSS04} and has since been extended to the broader framework of Clifford analysis \cite{BDS82,Gilbert91,LY24,CSSS04}.

However, unlike holomorphic functions, quaternionic regular functions are not closed under multiplication—a direct consequence of the non-commutativity of the quaternions. Specifically, if $f$ and $g$ are quaternionic regular, neither $fg$ nor $gf$ is necessarily regular. This fundamental obstruction has been a major challenge in the development of the theory.

The $k$-Cauchy–Fueter operator $\mathscr{D}^{(k)}$ (see Definition \eqref{def:kCF}) originates in physics as the elliptic version of spin-$k/2$ massless field operators over Minkowski spacetime \cite{CMW16,EPW80,PR84,PR86}. The $k$-Cauchy-Fueter operator on the higher dimensional quaternionic space was introduced  in \cite{Wang10} and functions annihilated by it are referred to as $k$-regular functions, which can be viewed as a generalization of  regular functions of several  quaternionic variables introduced by Pertici in \cite{Pertici88}. Indeed, there exists a one-to-one correspondence $\tau$ (see Definition \eqref{eq:def of tau}) between $1$-regular functions and quaternionic regular functions. Kang and Wang \cite{KW13} derived a Taylor expansion for $k$-regular functions via the Penrose integral formula. A systematic introduction to $k$-Cauchy–Fueter operators and $k$-regular functions can be found in \cite{Wang10,Wang15,Wang19,Wang25,RZ23}.

The central question addressed in this work is how to overcome the lack of multiplicative closure among regular functions in order to construct new ones. By examining $k$-regular polynomials, we introduce a pair of multiplication-like operators, $L_0$ and $L_1$ (see Section \ref{section:L0L1}), which generate $k$-regular functions from $(k+1)$-regular functions. Specifically, we establish the following result.

\begin{theorem}\label{thm:chengfa}
   Let $\Omega$ be a domain of $\mathbb{R}^4$, $k\geqslant1$ and $f,g$ be $k$-regular on $\Omega$. Then $L_0f+L_1g$ is a $(k-1)$-regular function. In particular,
    \begin{itemize}
    \item If $f,g$ are $2$-regular on $\Omega$, then the function $\tau(L_0f+L_1g)$ is quaternionic regular.
    \item If $f,g$ are $1$-regular on $\Omega$, then $L_0f+L_1g$ is harmonic.
    \end{itemize}
\end{theorem}

This result motivates the following inverse problem: given any complex-valued harmonic function $h$ on a domain $\Omega$, do there exist $1$-regular functions $f$ and $g$ such that $h=L_0f+L_1g$?

We establish a sufficient condition for this to be true.

\begin{theorem}\label{thm:inverse problem}
Let $\Omega$ be a domain of $\mathbb{R}^4$ with $H^3(\Omega,\mathbb{R})=0$ and $0\notin \Omega$. Then for any harmonic function $h\in C^\infty(\Omega,\mathbb{C})$, there is a pair of $1$-regular functions $f$ and $g$ such that $$h=L_0f+L_1g.$$
More precisely,
there exists $f_0,f_1,g_0,g_1\in C^\infty(\Omega,\mathbb{C})$ such that
\begin{equation}\label{eq:dec of h}
 h=(x_2+x_3i)f_0+(x_0+x_1i)f_1+(-x_0+x_1i)g_0+(x_2-x_3i)g_1
\end{equation}
and $D(\mathbf{j}f_0-f_1)=D(\mathbf{j}g_0-g_1)=0.$
\end{theorem}

As observed in \eqref{eq:dec of h}, the right-hand side equals zero at the origin, whereas the left-hand side does not necessarily vanish there. This is one reason why we require $0\notin \Omega$. More specifically, if we take $h=1$, Theorem \ref{thm:inverse problem} yields a decomposition of the constant function $1$ on certain domains excluding the origin.

It should also be noted that Theorem \ref{thm:inverse problem} imposes a topological condition on $\Omega$, namely $H^3(\Omega,\mathbb{R})=0$. This condition is used here to solve the $2$–Cauchy–Fueter equation. In fact, by constructing a new acyclic resolution of the sheaf $\mathcal{R}^{(2)}$ of $2$-regular functions (see Section \ref{section:A new resolution}), we provide a complete topological characterization of the solvability of the
$2$-Cauchy–Fueter equation.

 \begin{theorem}\label{thm:real valued harmonic function}
Let $\Omega\subset\mathbb{R}^4$ be a domain. Then the following are equivalent:
\begin{enumerate}[label=(\arabic*)]
  \item $H^3(\Omega,\mathbb{R})=0$.
  \item $H^1(\Omega,\mathcal{R}^{(2)})=0$. Equivalently, the $2$-Cauchy-Fueter equation is solvable on $\Omega$.
  \item For every real-valued harmonic function $h\in C^\infty(\Omega,\mathbb{R})$, there exists a quaternionic regular function $f\in C^\infty(\Omega,\mathbb{H})$  such that $\operatorname{Re}f=h$.
  \item $H^1(\Omega,\widetilde{\mathcal{R}} )=0,$ where $\widetilde{\mathcal{R}}$ is the subsheaf of quaternionic regular functions taking values in $\operatorname{Im}(\mathbb{H})$.
\end{enumerate}
\end{theorem}

 It is also noteworthy that the equivalent description $(3)$ can be viewed as a unification of the following two established theorems.
\begin{theorem}\label{thm:C andH}
\begin{enumerate}
  \item Let $\Omega\subset\mathbb{C}$ be a domain, for every real-valued harmonic function $h\in C^\infty(\Omega,\mathbb{R})$, there exists a holomorphic function $f\in C^\infty(\Omega,\mathbb{C})$  such that $\operatorname{Re}f=h$ if and only if $\Omega$ is simply connected.
  \item \cite{Nono86}Let $\Omega\subset\mathbb{H}$ be a domain, for every complex-valued harmonic function $h\in C^\infty(\Omega,\mathbb{C})$, there exists complex-valued harmonic function  $g\in C^\infty(\Omega,\mathbb{H})$  such that $h+\mathbf{j}g$ is quaternionic regular if and only if $\Omega$ is domain of holomorphy in a suitable variable change.
\end{enumerate}

\end{theorem}
The organization of this paper is as follows.
Section \ref{section:Preliminaries} covers the necessary background on the Cauchy–Fueter and $k$-Cauchy–Fueter operators. In Section \ref{section:L0L1}, we introduce and define the operators $L_0$ and $L_1$, and prove Theorem \ref{thm:chengfa}. Section \ref{section:A new resolution} presents a new resolution of $\mathcal{R}^{(2)}$, revealing the connection between the $2$-Cauchy–Fueter complex and the de Rham complex. The proof of the equivalent topological characterization for the $2$-Cauchy–Fueter equation (Theorem \ref{thm:real valued harmonic function}) is given in Section \ref{section:Proof of Theorem real valued harmonic function}. Finally, in Section \ref{section:Proof of Theorem inverse problem}, we prove the converse of Theorem \ref{thm:chengfa} (Theorem \ref{thm:inverse problem}).

\section{Preliminaries}\label{section:Preliminaries}

\subsection{Cauchy-Fueter operator}

  Let $\mathbb{H}$ be the quaternion space with the imaginary units $\mathbf{i,j,k}$ satisfying that
$$\mathbf{i}^2=\mathbf{j}^2=\mathbf{k}^2=\mathbf{ijk}=-1.$$
Denote an element $q$ of $\mathbb{H}$ by
$$q=x_0+x_1\mathbf{i}+x_2\mathbf{j}+x_3\mathbf{k}.$$
The conjugate of $q$ is defined by $\overline{q}=x_0-x_1\mathbf{i}-x_2\mathbf{j}-x_3\mathbf{k}$. Define the real part and  imaginary part of $q$ by
$$\operatorname{Re} q:=\frac{1}{2}(q+\bar{q})=x_0, \ \operatorname{Im} q:=x_1\mathbf{i}+x_2\mathbf{j}+x_3\mathbf{k}.$$

 The Cauchy-Fueter operator is defined as following
\begin{equation*}
    D = \frac{\partial}{\partial x_0}+\mathbf{i}\frac{\partial}{\partial x_1}+\mathbf{j}\frac{\partial}{\partial x_2}+\mathbf{k}\frac{\partial}{\partial x_3}.
\end{equation*}
Let $\Omega\subset\mathbb{R}^4$ be an open subset and $f\in C^1(\Omega,\mathbb{H})$ . If $Df=0$ on  $\Omega$, we call $f$ {\bf quaternionic regular} on $\Omega$. It is clear that $\Delta=D\overline{D}=\overline{D}D$, which implies that each component of a quaternionic regular function $f$ on $\Omega$ is harmonic.

  Let  $\mathcal{R}$ denote the sheaf of quaternionic regular functions on a domain $\Omega\subset \mathbb{R}^4$, and  $\widetilde{R}$ comprising the subsheaf of those functions whose real part vanishes.

\bigskip

The equation $Df=g$ is called Cauchy-Fueter equation. For the solution of Cauchy-Fueter equation, we have the following Lemma.
\begin{lemma}\label{lem:solution of Cauchy Fueter}\cite[Theorem 3.2.1]{CSSS04}
    Let $\Omega$ be a domain of $\mathbb{R}^4$ and $g\in C^\infty(\Omega,\mathbb{H})$. Then there exists $f\in C^\infty(\Omega,\mathbb{H})$ such that
    \begin{equation*}
        Df=g.
        \end{equation*}
        In addition, we have $H^1(\Omega,\mathcal{R})=0$.
\end{lemma}

\subsection{$k$-Cauchy-Fueter operators}

 As many properties of the $k$-Cauchy–Fueter equation and $k$-regular functions have already been investigated by Wang \cite{Wang10,Wang15,Wang19,Wang25,KW13}, here we briefly review the definition of the $k$-Cauchy–Fueter operator, along with basic notations and concepts such as the $k$-Cauchy–Fueter complex, $k$-regular functions.

We embed the quaternions ring $\mathbb{H}$ into the ring of complex $2\times 2$ matrices via
\begin{equation}\label{eq:z AAprime}
\tau(q):=\left(\begin{array}{cc}
    x_0+ x_1i & -x_2- x_3i  \\
     x_2- x_3i& x_0- x_1i
\end{array}\right)=
\left(\begin{array}{cc}
    \overline{z^{00^\prime}} & \overline{z^{01^\prime}}  \\
    \overline{z^{10^\prime}} & \overline{z^{11^\prime}}
\end{array}\right).
\end{equation}
This yields  operators  $\nabla_{AA^\prime}$,   defined by
\begin{equation}\label{eq:def of nabla AAprime}
  \left(\begin{array}{ll}
  \nabla_{00^\prime}   & \nabla_{01^\prime}\\
  \nabla_{10^\prime} & \nabla_{11^\prime}
  \end{array}\right):=2\left(\begin{array}{cc}
    \frac{\partial}{{\partial z^{00^\prime}}} & \frac{\partial}{{\partial z^{01^\prime}}}  \\
    \frac{\partial}{{\partial z^{10^\prime}}} & \frac{\partial}{{\partial z^{11^\prime}}}
\end{array}\right)=
\left(
\begin{array}{ll}
\dfrac{\partial}{\partial  x_{0} }+i\dfrac{\partial}{\partial x_{1}}   & -\dfrac{\partial}{\partial  x_{2} }-i\dfrac{\partial}{\partial x_{3}} \\
 \dfrac{\partial}{\partial  x_{2} }-i\dfrac{\partial}{\partial x_{3}} &\  \ \ \dfrac{\partial}{\partial  x_{0} }-i\dfrac{\partial}{\partial x_{1}}
\end{array}
\right).
\end{equation}

We can write smooth quaternions valued functions $u,f$ as $u=u_0+\mathbf{j}u_1$ and $f=f_0+\mathbf{j}f_1$ for complex valued functions $u_0,u_1,f_0$ and $f_1$. Then the Cauchy-Fueter equation
$$Du=f$$
is equivalent to the following equation
\begin{equation*}
    \left(\begin{array}{ll}
  \nabla_{00^\prime}   & \nabla_{01^\prime}\\
  \nabla_{10^\prime} & \nabla_{11^\prime}
  \end{array}\right)
  \left(\begin{array}{l}
       u_0 \\
       u_1
  \end{array}\right)=
  \left(\begin{array}{l}
       f_0 \\
       f_1
  \end{array}\right).
\end{equation*}

By raising  indices, we introduce derivatives
$$\nabla^{A^\prime}_{A}:=\sum_{B^\prime=0, 1} \nabla_{AB^\prime} \ \epsilon^{B^\prime A^\prime},$$
  where
 \begin{equation}\label{eq:def epsilon}
\epsilon=(\epsilon^{B^\prime A^\prime})=\left(\begin{array}{cc}
  0 & -1 \\
  1 & 0
\end{array}\right)=\tau(\mathbf{j}).
\end{equation}
 More precisely, we have
\begin{equation}\label{preliminary:nabla AA prime}
  \left(\begin{array}{ll}
  \nabla_{0}^{0^\prime}   & \nabla_{0}^{1^\prime}\\
  \nabla_{1}^{0^\prime} & \nabla_{1}^{1^\prime}
  \end{array}\right)=
\left(
\begin{array}{ll}
-\dfrac{\partial}{\partial x_{2}}-i\dfrac{\partial}{\partial x_{3}} & -\dfrac{\partial}{\partial x_{0}}-i\dfrac{\partial}{\partial x_{1}} \\
\ \ \ \dfrac{\partial}{\partial x_{0}}-i\dfrac{\partial}{\partial x_{1}}   & -\dfrac{\partial}{\partial x_{2}}+i\dfrac{\partial}{\partial x_{3}}
\end{array}
\right).
\end{equation}

Before the introduction of $k$-Cauchy-Fueter operator, we first establish some necessary notations. An element of $\mathbb{C}^2$ is denoted by
$$(f_{A^\prime}), \qquad {A^\prime=0^\prime,1^\prime},$$   where $f_{A^\prime}\in\mathbb{C}$. The symmetric power $\odot^k\mathbb{C}^2$ is a subspace of $\otimes^k\mathbb{C}^2$ whose element can be expressed by  \begin{equation*}
(f_{ A_0^\prime\cdots A_{k-1}^\prime}), \qquad A_0^\prime,\ldots,A_{k-1}^\prime=0,1,
  \end{equation*}
  where $f_{A_0^\prime\ldots A_{k-1}^\prime}\in\mathbb{C}$ are invariant under the permutation of subscripts, i.e.
$$f_{A_0^\prime\ldots A_{k-1}^\prime}=f_{A^\prime_{\sigma(0)}\ldots A^\prime_{\sigma(k-1)}}$$
for any $\sigma$ in the group $S_k$ of permutations of $k$ letters.

For the specific case $k=0$, we write $\odot^0\mathbb{C}^2\cong \mathbb{C}$.

An element of $\mathbb{C}^2\otimes\odot^k\mathbb{C}^2$ is denoted by
\begin{equation*}
(f_{ AA_0^\prime\cdots A_{k-1}^\prime}),\qquad  A,A_0^\prime,\ldots,A_{k-1}^\prime=0,1,
\end{equation*}
 where $f_{AA_0^\prime\ldots A_{k-1}^\prime}\in\mathbb{C}$ are invariant under the permutation of subscripts $A_0^\prime\ldots A_{k-1}^\prime$.

\begin{example}
Let $k=2$. An element $f\in \odot^2\mathbb{C}^2$ is formally given by a vector $(f_{0^\prime0^\prime},f_{0^\prime1^\prime},f_{1^\prime0^\prime},f_{1^\prime1^\prime})$ which satisfies  $f_{0^\prime1^{\prime}}=f_{1^\prime0^\prime}$.
\end{example}

Now we introduce the $k$-Cauchy-Fueter operators
\begin{equation*}
\begin{aligned}
  \mathscr{D}^{(k)}: & C^\infty(\mathbb{R}^{4},\odot^k\mathbb{C}^2)\longrightarrow C^\infty(\mathbb{R}^{4},\mathbb{C}^{2}\otimes\odot^{k-1}\mathbb{C}^2),
\end{aligned}
\end{equation*}
defined by
 \begin{equation}\label{def:kCF}
   \left(\mathscr{D}^{(k)}f\right)_{AA_0^\prime\cdots A_{k-2}^\prime} :=
  \begin{cases} \sum\limits_{A^\prime=0^\prime,1^\prime}\nabla^{A^\prime}_Af_{A^\prime},  & \qquad k=1,\\  \\ \sum\limits_{A^\prime=0^\prime,1^\prime}\nabla^{A^\prime}_Af_{A^\prime A_0^\prime\cdots A_{k-2}^\prime}, & \qquad   k\geqslant2.
  \end{cases}
\end{equation}
Let $\Omega\subset\mathbb{R}^4$ be an open subset and $f\in C^1(\Omega,\odot^k\mathbb{C}^2)$ . If $\mathscr{D}^{(k)}f=0$ on  $\Omega$, we call $f$ {\bf k-regular} on $\Omega$.  If $f$ is harmonic on $\Omega$, then we call $f$ $0$-regular on $\Omega$.

\begin{example}
$(1)$  Let $\Omega\subset\mathbb{R}^4$ be an open subset. Then function $f\in C^1(\Omega,\mathbb{C}^2)$ is a $1$-regular function on $\Omega$ if and only if
\begin{equation}\label{eq:def of tau}
\tau(f):=\mathbf{j}(f_0+\mathbf{j}f_1)
\end{equation}
is a quaternionic regular function on $\Omega$.

$(2)$ We can provide an explicit expression for $2$-regular functions. If $f\in C^\infty(\Omega,\odot^2\mathbb{C}^2)$ is $2$-regular, then $f$ satisfies the following equations
\begin{equation*}
\begin{aligned}
&\nabla_{0}^{0^\prime}f_{0^\prime0^\prime}+\nabla_{0}^{1^\prime}f_{1^\prime0^\prime}=0,\qquad \nabla_{1}^{0^\prime}f_{0^\prime0^\prime}+\nabla_{1}^{1^\prime}f_{1^\prime0^\prime}=0,\\
&\nabla_{0}^{0^\prime}f_{0^\prime1^\prime}+\nabla_{0}^{1^\prime}f_{1^\prime1^\prime}=0,\qquad \nabla_{1}^{0^\prime}f_{0^\prime1^\prime}+\nabla_{1}^{1^\prime}f_{1^\prime1^\prime}=0, \qquad f_{0^\prime1^\prime}=f_{1^\prime0^\prime}.
\end{aligned}
\end{equation*}

$(3)$ Consider a quaternionic regular function $F=F_1\mathbf{i}+F_2\mathbf{j}+F_3\mathbf{k}$ with real-valued components $F_1,F_2,F_3$. We define a smooth function $f\in C^\infty(\Omega,\odot^2\mathbb{C}^2)$ by
\[f_{0^\prime0^\prime}=F_2-iF_3, \qquad  f_{0^\prime1^\prime}=f_{1^\prime0^\prime}=-iF_1,   \qquad f_{1^\prime1^\prime}=F_2+iF_3.\]
Then, $f$ is 2-regular.
\end{example}

\bigskip

Let $k\geqslant2$. Denote by $\mathcal{R}^{(k)}$ the sheaf of $k$-regular functions over domain $\Omega\subset\mathbb{R}^4$. Consider the complex vector spaces
$$V^{(k)}_0=\odot^k\mathbb{C}^2,\qquad V_1^{(k)}=\odot^{k-1}\mathbb{C}^2\otimes \mathbb{C}^2,\qquad V_2^{(k)}=\odot^{k-2}\mathbb{C}^2\otimes \Lambda^2\mathbb{C}^2\cong\odot^{k-2}\mathbb{C}^2.$$
For fixed $\alpha=0,1,2$, denote by $E_\alpha^{(k)}$ the trivial bundle over $\Omega$ whose fiber is  $V_\alpha^{(k)}$. Denote by $\mathcal{V}_\alpha^{(k)}$ the  sheaf of $E_\alpha^{(k)}$. There is a acyclic resolution of the sheaf  $\mathcal{R}^{(k)}$ \cite{Wang10}:
\begin{equation}\label{eq:resolution of k CF}
\begin{split}
0\rightarrow \mathcal{R}^{(k)}\xrightarrow {i}\mathcal{V}_0^{(k)} &
\xrightarrow{{\mathscr D}^{(k)}}\mathcal{V}_1^{(k)}
\xrightarrow{{\mathscr D}_1^{(k)}}\mathcal{V}_2^{(k)}\longrightarrow 0 .\end{split}
\end{equation}
Here $i$ is the embedding operator and the operator
$$\mathscr{D}_1^{(k)}  : C^\infty(\mathbb{R}^{4},\mathbb{C}^2\otimes\odot^{k-1}\mathbb{C}^2)\longrightarrow  C^\infty(\mathbb{R}^{4},\odot^{k-2}\mathbb{C}^2),$$
is defined by
\begin{equation}\label{eq:def of D1}
(\mathscr{D}_1^{(k)}f)_{A_0^\prime\ldots A_{k-3}^\prime}=\begin{cases} \sum\limits_{A^\prime=0^\prime,1^\prime}(\nabla^{A^\prime}_0f_{1A^\prime}-\nabla^{A^\prime}_1f_{0A^\prime}),  & \qquad k=2,\\  \\ \sum\limits_{A^\prime=0^\prime,1^\prime}(\nabla^{A^\prime}_0f_{1A^\prime A_0^\prime\cdots A_{k-3}^\prime}-\nabla^{A^\prime}_1f_{0A^\prime A_0^\prime\cdots A_{k-3}^\prime}), & \qquad   k\geqslant3.
  \end{cases}
\end{equation}

Since the resolution \eqref{eq:resolution of k CF} is acyclic, we have
$$H^1(\Omega,\mathcal{R}^{(k)})\cong\frac{\mathop{Ker}\mathscr{D}_1^{(k)}}{\mathop{Im}\mathscr{D}^{(k)}}. $$

\section{The operators $L_0$ and $L_1$}\label{section:L0L1}

 In this section, we introduce a pair of multiplication-like operator $L_0,L_1$ and provide a proof of Theorem \ref{thm:chengfa}.

Sudbery \cite{Sudbery79} observed that if $f$ is a quaternionic regular, then $qf$ is harmonic. Similarly, for a $1$-regular function
$(f_0,f_1)$,  the function $q\mathbf{j}\cdot (f_0+\mathbf{j}f_1)$ is also harmonic. Consider the embedding of $-q\mathbf{j}$ from $\mathbb{H}$ to $\mathbb{C}^{2\times 2}$, defined by
\begin{equation}\label{eq:def:zAAprime}
\tau(-q\mathbf{j})=\tau(x_2+x_3\mathbf{i}-x_0\mathbf{j}-x_1\mathbf{k})=\left(
\begin{array}{cc}
    x_2+ x_3 i & x_0+ x_1i \\
    -x_0+ x_1i & x_2- x_3i
\end{array}
\right)=\left(\begin{array}{cc}
   z^{10^\prime} & z^{11^\prime} \\
   -z^{00^\prime} & -z^{01^\prime}
\end{array}\right).
\end{equation}
We denote the matrix $\tau(-q\mathbf{j})$ by $(z_A^{A^\prime})$, where each element $z_A^{A^\prime}$ satisfies the equation
\begin{equation}\label{eq:zAAprime calculation}
  z^{AA^\prime}=\sum_{B=0,1}\epsilon^{AB}z_B^{A^\prime}.
\end{equation}
Here $\epsilon$ denote the matrix defined in equation \eqref{eq:def epsilon}.

  Then functions $z_0^{0^\prime}f_0+z_0^{1^\prime}f_1$ and $z_1^{0^\prime}f_0+z_1^{1^\prime}f_1$ are harmonic whenever $(f_0,f_1)$ is $1$-regular. Motivated by these facts, we define a pair of operators as follows:
\begin{equation}\label{eq:def of L0L1}
\begin{aligned}
    L_0(f_0,f_1)&:=z_0^{0^\prime}f_0+z_0^{1^\prime}f_1=(x_2+x_3i)f_0+(x_0+x_1i)f_1,\\
    L_1(f_0,f_1)&:=z_1^{0^\prime}f_0+z_1^{1^\prime}f_1=(-x_0+x_1i)f_0+(x_2-x_3i)f_1.
    \end{aligned}
\end{equation}

 The definition of operators $L_0$ and $L_1$ can be naturally extended to the spaces $C^\infty(\Omega,\odot^k\mathbb{C}^2)$ and $C^\infty(\Omega,\mathbb{C}^2\otimes\odot^k\mathbb{C}^2)$.
\begin{definition}\label{def:L_0L_1 on k-CF}
    Let $\Omega$ be a domain of $\mathbb{R}^4$. Operators $L_0,L_1$ are defined on $C^\infty(\Omega,\odot^k\mathbb{C}^2)$ as follows:
\begin{equation*}
    \begin{aligned}
&L_0,L_1:C^\infty(\Omega,\odot^k\mathbb{C}^{2})\longrightarrow C^\infty(\Omega,\odot^{k-1}\mathbb{C}^2),\\
&(L_0f)_{A^\prime_0\ldots A^\prime_{k-2}}:=z_0^{0^\prime} f_{0^\prime A^\prime_0\ldots  A_{k-2}^\prime}+z_0^{1^\prime} f_{1^\prime A^\prime_0\ldots  A_{k-2}^\prime},\\
&(L_1f)_{A^\prime_0\ldots A^\prime_{k-2}}:=z_1^{0^\prime} f_{0^\prime A^\prime_0\ldots  A_{k-2}^\prime}+z_1^{1^\prime} f_{1^\prime A^\prime_0\ldots  A_{k-2}^\prime}.
    \end{aligned}
\end{equation*}
Similarly, on the space $C^\infty(\Omega,\mathbb{C}^2\otimes\odot^k\mathbb{C}^2 )$, operators $L_0,L_1$ are defined by:
\begin{equation*}
    \begin{aligned}
&L_0,L_1:C^\infty(\Omega,\mathbb{C}^2\otimes\odot^k\mathbb{C}^{2})\longrightarrow C^\infty(\Omega,\mathbb{C}^2\otimes\odot^{k-1}\mathbb{C}^2),\\
&(L_0f)_{AA^\prime_0\ldots A^\prime_{k-2}}:=z_0^{0^\prime} f_{A0^\prime A^\prime_0\ldots  A_{k-2}^\prime}+z_0^{1^\prime} f_{A1^\prime A^\prime_0\ldots  A_{k-2}^\prime},\\
&(L_1f)_{AA^\prime_0\ldots A^\prime_{k-2}}:=z_1^{0^\prime} f_{A0^\prime A^\prime_0\ldots  A_{k-2}^\prime}+z_1^{1^\prime} f_{A1^\prime A^\prime_0\ldots  A_{k-2}^\prime}.
    \end{aligned}
\end{equation*}
\end{definition}

The following properties of operators $L_0$ and $L_1$ are important.
 \begin{proposition}\label{prop:L0L1}
 Operator $L_0$ and $L_1$ satisfy the following properties:
\begin{enumerate}
\item $L_0L_1=L_1L_0$.
\item For $j=0,1$ and any $k\geqslant 2$, the following intertwining relation holds:
\[\mathscr{D}^{(k-1)}L_j=L_j\mathscr{D}^{(k)}.\]
\end{enumerate}
 \end{proposition}

Before presenting the proof, we first state a useful lemma for subsequent computations.
\begin{lemma}\label{lem:computation z}\cite[Lemma 3.3]{Wang17}
\begin{enumerate}
\item $\nabla_{AA^\prime} z^{BB^\prime}=2\delta_A^{B}\delta_{A^\prime}^{B^\prime}$.
\item $\nabla_A^{A^\prime}z_{B}^{B^\prime}=-2\epsilon^{B^\prime A^\prime}\epsilon^{BA}$.
\end{enumerate}
\end{lemma}
\begin{proof}
The assertion $(1)$ corresponds precisely to Lemma $3.3$ in \cite{Wang17}. Thus, we only need to prove assertion $(2)$. From the definitions of $z_A^{A^\prime}$ (see \eqref{eq:zAAprime calculation}) and $\nabla_A^{A^\prime}$ (see \eqref{preliminary:nabla AA prime}), we recall that
\[
\nabla_A^{A^\prime}z_{B}^{B^\prime}=\sum_{C^\prime,D=0,1}\nabla_{AC^\prime}\left(\epsilon^{C^\prime A^\prime}(\epsilon^{-1})^{ BD }z^{DB^\prime}\right).
\]
By the definition of the matrix $\epsilon$ (see \eqref{eq:def epsilon}), we have $\epsilon^{-1}=-\epsilon$. Utilizing assertion $(1)$, we derive the following:
\begin{align*}
\nabla_A^{A^\prime}z_{B}^{B^\prime}&=-\sum_{C^\prime,D=0,1}\left(\nabla_{AC^\prime}z^{DB^\prime}\right)\epsilon^{C^\prime A^\prime}\epsilon^{BD}\\
&= -2\sum_{C^\prime,D=0,1}\delta_A^D\delta^{B^\prime}_{C^\prime}\epsilon^{C^\prime A^\prime}\epsilon^{BD}= -2\epsilon^{B^\prime A^\prime}\epsilon^{BA}.
\end{align*}
\end{proof}

\emph{Proof of Proposition \ref{prop:L0L1}.}
We first prove part $(1)$. Let $f\in C^\infty(\Omega,\odot^2\mathbb{C}^2)$. Then we have
\begin{equation}\label{eq:prop L0L1:proof of 1:1}
\begin{aligned}
    L_0L_1f&=z_0^{0^\prime}(L_1f)_{0^\prime}+z_0^{1^\prime}(L_1f)_{1^\prime}\\
    &=z_0^{0^\prime}(z_1^{0^\prime} f_{0^\prime0^\prime}+z_1^{1^\prime} f_{1^\prime0^\prime}) +z_0^{1^\prime}(z_1^{0^\prime} f_{0^\prime1^\prime}+z_1^{1^\prime} f_{1^\prime1^\prime})\\
    &=z_1^{0^\prime}(z_0^{0^\prime}f_{0^\prime0^\prime}+z_0^{1^\prime}f_{1^\prime0^\prime})+z_{1}^{1^\prime}(z_0^{0^\prime}f_{0^\prime1^\prime}+z_0^{1^\prime}f_{1^\prime1^\prime})\\
    &=z_1^{0^\prime}(L_0f)_{0^\prime}+z_1^{1^\prime}(L_0f)_{1^\prime}=L_1L_0f.
    \end{aligned}
\end{equation}
This establishes the result for the case $k=2$.

The proof of cases $k\geqslant 3$ and $f\in C^\infty(\Omega,\odot^k\mathbb{C}^2)$ follows from the applications of \eqref{eq:prop L0L1:proof of 1:1} on the term $(L_0L_1 f)_{A_0^\prime\ldots A_{k-3}^\prime}$. More precisely, we have
\begin{equation*}
\begin{aligned}
    (L_0L_1f)_{A_0^\prime\ldots A_{k-3}^\prime}&=z_0^{0^\prime}(z_1^{0^\prime} f_{0^\prime0^\prime A_0^\prime\ldots A_{k-3}^\prime}
    +z_1^{1^\prime} f_{1^\prime0^\prime A_0^\prime\ldots A_{k-3}^\prime})\\
    &\quad +z_0^{1^\prime}(z_1^{0^\prime} f_{0^\prime1^\prime A_0^\prime\ldots A_{k-3}^\prime}+z_1^{1^\prime} f_{1^\prime1^\prime A_0^\prime\ldots A_{k-3}^\prime})\\
    &=z_1^{0^\prime}(z_0^{0^\prime}f_{0^\prime0^\prime A_0^\prime\ldots A_{k-3}^\prime}+z_0^{1^\prime}f_{1^\prime0^\prime A_0^\prime\ldots A_{k-3}^\prime})\\
    &\quad +z_{1}^{1^\prime}(z_0^{0^\prime}f_{0^\prime1^\prime A_0^\prime\ldots A_{k-3}^\prime}
    +  z_0^{1^\prime}f_{1^\prime1^\prime A_0^\prime\ldots A_{k-3}^\prime})\\
    &=(L_1L_0f)_{A_0^\prime\ldots A_{k-3}^\prime}.
    \end{aligned}
\end{equation*}
 The proof for the cases $k\geqslant 2$ and $f\in C^\infty(\Omega,\mathbb{C}^2\otimes\odot^k\mathbb{C}^2)$ proceeds similarly.

\bigskip

Now we prove the assertion $(2)$. Consider the case $k=2$ firstly. Assume that
$$f\in C^\infty(\Omega,\odot^2\mathbb{C}^2).$$
 A direct computation shows that
\begin{equation}\label{eq:prop L0L1:proof of 2:1}
    \begin{aligned}
   (\mathscr{D}^{(1)}(L_jf))_A&=\nabla_A^{0^\prime}(L_jf)_{0^\prime}+\nabla_A^{1^\prime}(L_jf)_{1^\prime}\\
   &=\nabla_A^{0^\prime}(z_j^{0^\prime}f_{0^\prime0^\prime}+z_j^{1^\prime}f_{1^\prime0^\prime})+\nabla_A^{1^\prime}(z_j^{0^\prime}f_{0^\prime1^\prime}+z_j^{1^\prime}f_{1^\prime1^\prime}).
    \end{aligned}
\end{equation}
Here $j=0,1$.
Then equations \eqref{eq:prop L0L1:proof of 2:1} and Lemma \ref{lem:computation z} imply that
\begin{equation*}
\begin{aligned}
     (\mathscr{D}^{(1)}(L_jf))_A&=z_j^{0^\prime}\nabla_A^{0^\prime}f_{0^\prime0^\prime}+z_j^{1^\prime}\nabla_A^{0^\prime}f_{1^\prime0^\prime} +z_j^{0^\prime}\nabla_A^{1^\prime}f_{0^\prime1^\prime}+z_j^{1^\prime}\nabla_A^{1^\prime}f_{1^\prime1^\prime}\\
     &\quad -2\epsilon^{jA}\epsilon^{10}f_{1^\prime0^\prime}-2\epsilon^{jA}\epsilon^{01}f_{0^\prime1^\prime}\\
     &=z_j^{0^\prime}\nabla_A^{0^\prime}f_{0^\prime0^\prime}+z_j^{1^\prime}\nabla_A^{0^\prime}f_{1^\prime0^\prime} +z_j^{0^\prime}\nabla_A^{1^\prime}f_{0^\prime1^\prime}+z_j^{1^\prime}\nabla_A^{1^\prime}f_{1^\prime1^\prime}\\
     &=z_j^{0^\prime}(\mathscr{D}^{(2)}f)_{A0^\prime}+z_j^{1^\prime}(\mathscr{D}^{(2)}f)_{A1^\prime}=(L_j\mathscr{D}^{(2)}f)_A.
     \end{aligned}
\end{equation*}
This proves the result for the case $k=2$.

Now we assume that $f\in C^\infty(\Omega,\odot^k\mathbb{C}^2)$ for some $k>2$. It follows from direct computation that
\begin{equation*}
\begin{aligned}
    (\mathscr{D}^{(k-1)}(L_jf))_{AA_0^\prime\ldots A_{k-3}^\prime}&=z_j^{0^\prime}\nabla_A^{0^\prime}f_{0^\prime0^\prime A_0^\prime\ldots A_{k-3}^\prime}+z_j^{1^\prime}\nabla_A^{0^\prime}f_{1^\prime0^\prime A_0^\prime\ldots A_{k-3}^\prime} \\
    &\quad +z_j^{0^\prime}\nabla_A^{1^\prime}f_{0^\prime1^\prime A_0^\prime\ldots A_{k-3}^\prime}+z_j^{1^\prime}\nabla_A^{1^\prime}f_{1^\prime1^\prime A_0^\prime\ldots A_{k-3}^\prime}\\
    &\quad -2\epsilon^{jA}\epsilon^{10}f_{1^\prime0^\prime A_0^\prime\ldots A_{k-3}^\prime}-2\epsilon^{jA}\epsilon^{01}f_{0^\prime1^\prime A_0^\prime\ldots A_{k-3}^\prime}\\
    &=z_j^{0^\prime}(\mathscr{D}^{(k)}f)_{A0^\prime A_0^\prime\ldots A_{k-3}^\prime}+z_j^{1^\prime}(\mathscr{D}^{(k)}f)_{A1^\prime A_0^\prime\ldots A_{k-3}^\prime}\\
    &=(L_j\mathscr{D}^{(k)}f)_{A A_0^\prime\ldots A_{k-3}^\prime}.\\
    \end{aligned}
\end{equation*}
Here $A_0^\prime,\ldots, A_{k-3}^\prime=0^\prime,1^\prime$. This completes the proof.\qed

 We conclude this section by proving Theorem \ref{thm:chengfa}.

 \emph{Proof of Theorem \ref{thm:chengfa}.} We begin with the case where $k>0$ and $f,g$ are $(k+1)$-regular functions. It follows from Proposition \ref{prop:L0L1} that
\[\mathscr{D}^{(k)}(L_0f+L_1g)=L_0\mathscr{D}^{(k+1)}f+L_1\mathscr{D}^{(k+1)}g=0.\]
This proves the case $k>0$. The proof of case $k=0$ follows directly from the results of Sudbery \cite{Sudbery79}. \qed

\section{A new resolution of sheaf $\mathcal{R}^{(2)}$}\label{section:A new resolution}
In this section, we present a new resolution of the sheaf $\mathcal{R}^{(2)}$ to  reveal  the relationship between its cohomology groups and de Rham cohomology. Before constructing the resolution, we first establish the necessary notations and preliminary lemmas.

Let $\ast$ denote the Hodge star operator associated with the Euclidean metric. A complex valued $2$-form $f$ is said to be self-dual if $\ast f=f$ and anti-self-dual if $\ast f=-f$. Let $\Lambda^2_+$ and $\Lambda^2_-$ be the vector bundles of self-dual and anti-self-dual forms over $\mathbb{R}^4$, respectively. We focus on the anti-self-dual bundle $\Lambda^2_-$, which is a trivial rank-3 bundle over $\mathbb{R}^4$. A trivialization is given by the following global sections:
\begin{equation}\label{eq:three section of LJ}
   dx_0\wedge dx_1-dx_2\wedge dx_3,\qquad dx_0\wedge dx_2+dx_1\wedge dx_3,\quad \mbox{and} \quad dx_0\wedge dx_3-dx_1\wedge dx_2.
\end{equation}

The following lemma is crucial for this article.

\begin{lemma}\label{lem:essential}
 Let $U\subset \mathbb{R}^4$ be a domain. If $H^3(U,\mathbb{R})=0$, then for any complex valued closed $3$-form $g$ on $U$, there exists $f\in C^\infty(U,\LJ)$ such that $df=g$.
\end{lemma}

\begin{proof}
Let $g$ be a complex valued closed $3$-form on $U$. Then there exists complex valued $2$-form $u$ such that $du=g$.
Denote by
\begin{equation*}
     u=u_{01}dx_0\wedge dx_1+u_{02}dx_0\wedge dx_2+u_{03}dx_0\wedge dx_3+u_{12}dx_1\wedge dx_2+u_{13}dx_1\wedge dx_3+u_{23}dx_2\wedge dx_3.
\end{equation*}
Note that $d(u+dF)=g$ holds for any  $1$-form $F$. Our goal is to find a $1$-form
$$F=F_0dx_0+F_1dx_1+F_2dx_2+F_3dx_3$$
such that
\begin{equation}\label{eq:essential lemma:1}
    u+dF\in C^\infty(U,\LJ).
\end{equation}

It is easily verified by direct computation that
\begin{equation*}
  \begin{aligned}
     u+dF&=\left(u_{01}+ \frac{\partial}{\partial x_0}F_1-\frac{\partial}{\partial x_1}F_0 \right)dx_0\wedge dx_1+\left(u_{02}+ \frac{\partial}{\partial x_0}F_2-\frac{\partial}{\partial x_2}F_0 \right)dx_0\wedge dx_2\\
     &\quad +\left(u_{03}+ \frac{\partial}{\partial x_0}F_3-\frac{\partial}{\partial x_3}F_0 \right)dx_0\wedge dx_3+\left(u_{12}+ \frac{\partial}{\partial x_1}F_2-\frac{\partial}{\partial x_2}F_1 \right)dx_1\wedge dx_2\\
     &\quad +\left(u_{13}+ \frac{\partial}{\partial x_1}F_3-\frac{\partial}{\partial x_3}F_1 \right)dx_1\wedge dx_3+\left(u_{23}+ \frac{\partial}{\partial x_2}F_3-\frac{\partial}{\partial x_3}F_2 \right)dx_2\wedge dx_3.
     \end{aligned}
\end{equation*}
Then equation \eqref{eq:essential lemma:1} is equivalent to the following equations:
\begin{equation}\label{eq:essential lemma: equivalent equations}
   \left\{ \begin{array}{l}
          -\dfrac{\partial}{\partial x_1}F_0+\dfrac{\partial}{\partial x_0}F_1-\dfrac{\partial}{\partial x_3}F_2+\dfrac{\partial}{\partial x_2}F_3=-u_{01}-u_{23},\\
         -\dfrac{\partial}{\partial x_2}F_0+\dfrac{\partial}{\partial x_3}F_1+\dfrac{\partial}{\partial x_0}F_2-\dfrac{\partial}{\partial x_1}F_3=u_{13}-u_{02},\\
         -\dfrac{\partial}{\partial x_3}F_0-\dfrac{\partial}{\partial x_2}F_1+\dfrac{\partial}{\partial x_1}F_2+ \dfrac{\partial}{\partial x_0}F_3=-u_{03}-u_{12}.
    \end{array}\right.
\end{equation}

We claim that the solution of \eqref{eq:essential lemma: equivalent equations}
is closely related to the solution of Cauchy-Fueter equation.  To establish this relation, we  denote the real and imaginary parts of 2-form $u$ by $u^{(1)}$ and $u^{(2)}$, respectively.
Let $F^{(1)}$ and $F^{(2)}$ be solutions of Cauchy-Fueter equations
\begin{equation}\label{eq:essential lemma: Cauchy Fueter equations}
\begin{aligned}
     DF^{(1)}=-(u_{01}^{(1)}+u_{23}^{(1)})\mathbf{i}+(u_{13}^{(1)}-u_{02}^{(1)})\mathbf{j}-(u_{03}^{(1)}+u_{12}^{(1)})\mathbf{k},\\
      DF^{(2)}=-(u_{01}^{(2)}+u_{23}^{(2)})\mathbf{i}+(u_{13}^{(2)}-u_{02}^{(2)})\mathbf{j}-(u_{03}^{(2)}+u_{12}^{(2)})\mathbf{k},
     \end{aligned}
\end{equation}
respectively. Lemma \ref{lem:solution of Cauchy Fueter} implies that equations \eqref{eq:essential lemma: Cauchy Fueter equations} are always solvable on domain $U$.  We can express equations \eqref{eq:essential lemma: Cauchy Fueter equations}
as a system of equations as following:
\begin{equation}\label{eq:essential lemma: main}
   \left\{ \begin{array}{l}
    \dfrac{\partial}{\partial x_0}F^{(j)}_0-\dfrac{\partial}{\partial x_1}F^{(j)}_1-\dfrac{\partial}{\partial x_2}F^{(j)}_2-\dfrac{\partial}{\partial x_3}F^{(j)}_3=0,\\
          \dfrac{\partial}{\partial x_1}F^{(j)}_0+\dfrac{\partial}{\partial x_0}F^{(j)}_1-\dfrac{\partial}{\partial x_3}F^{(j)}_2+\dfrac{\partial}{\partial x_2}F^{(j)}_3=-u^{(j)}_{01}-u^{(j)}_{23},\\
         \dfrac{\partial}{\partial x_2}F^{(j)}_0+\dfrac{\partial}{\partial x_3}F^{(j)}_1+\dfrac{\partial}{\partial x_0}F^{(j)}_2-\dfrac{\partial}{\partial x_1}F^{(j)}_3=u^{(j)}_{13}-u^{(j)}_{02},\\
         \dfrac{\partial}{\partial x_3}F^{(j)}_0-\dfrac{\partial}{\partial x_2}F^{(j)}_1+\dfrac{\partial}{\partial x_1}F^{(j)}_2+ \dfrac{\partial}{\partial x_0}F^{(j)}_3=-u^{(j)}_{03}-u^{(j)}_{12}.
    \end{array}\right.
\end{equation}
Here $F_i^{(j)},i=0,1,2,3,j=1,2,$ are real valued functions and
\begin{equation*}
     F^{(j)}=F_0^{(j)}+F_1^{(j)}\mathbf{i}+F_2^{(j)}\mathbf{j}+F_3^{(j)}\mathbf{k},\qquad j=1,2.
\end{equation*}

Then we can take 1-form $\tilde{F}$ as following:
\begin{equation*}
     \tilde{F}_0=-F_0^{(1)}-iF_0^{(2)},\qquad \tilde{F}_1=F_1^{1}+iF_1^{(2)}, \qquad \tilde{F}_2=F_2^{1}+iF_2^{(2)}, \qquad \tilde{F}_3=F_3^{1}+iF_3^{(2)}.
\end{equation*}
It follows from equation \eqref{eq:essential lemma: main} that $\tilde{F}$ is a solution of equation  \eqref{eq:essential lemma: equivalent equations}. This completes the proof.
\end{proof}

We now present a new resolution of the sheaf $\mathcal{R}^{(2)}$.
\begin{proposition}\label{prop:a new resolution}
Let $\Omega\subset\mathbb{R}^4$ be a domain. Denote by $\mathcal{E}^i$ the sheaf of complex-valued differential $i$-forms over $\Omega$. Denote by $\mathcal{E}^2_{-}$ the sheaf of sections of the vector bundle $\LJ$ over $\Omega$.
We have the following acyclic resolution of the sheaf $\mathcal{R}^{(2)}$:
\begin{equation}\label{eq:another resolution}
    \begin{split}
0\rightarrow \mathcal{R}^{(2)}\xrightarrow {\eta}\mathcal{E}^2_{-} &
\xrightarrow{d}\mathcal{E}^3
\xrightarrow{d}\mathcal{E}^4\rightarrow 0.\end{split}
\end{equation}
Here $d$ is the exterior differentiation operator and $\eta$ is defined by
\begin{equation}\label{eq:def of eta}
     \eta_U(f):=f_{0^\prime0^\prime}d\bar{z}_0^{1^\prime}\wedge d\bar{z}_1^{1^\prime}+f_{0^\prime1^\prime}(-d\bar{z}_0^{0^\prime}\wedge d\bar{z}_1^{1^\prime}+d\bar{z}_1^{0^\prime}\wedge d\bar{z}_0^{1^\prime})-f_{1^\prime1^\prime}d\bar{z}_1^{0^\prime}\wedge d\bar{z}_0^{0^\prime}
\end{equation}
for any open set $U\subset\Omega$ and any $f\in C^\infty(U,\odot^{2}\mathbb{C}^2)$. Hence we have

\begin{equation}\label{eq:H_1}
  H^1(\Omega,\mathcal{R}^{(2)})\cong\frac{\mathop{Ker}\{d:\mathcal{E}^3(\Omega)\to \mathcal{E}^4(\Omega)\}}{\mathop{Im}\{d:\mathcal{E}^2_-(\Omega)\to \mathcal{E}^3(\Omega) \}}.
\end{equation}

\end{proposition}
\emph{Proof.}
Firstly, we will show that the image of $\eta_U$ belongs to $C^\infty(U,\LJ)$ for fixed open subset $U$. Recall
the definition of $z_A^{A^\prime}$ \eqref{eq:def:zAAprime}. It follows from direct computation that
\begin{equation}\label{eq:dx dz dengjia:1}
  \begin{aligned}
   f_{0^\prime0^\prime}d\bar{z}_0^{1^\prime}\wedge d\bar{z}_1^{1^\prime}&=f_{0^\prime0^\prime}(-dx_0+idx_1)\wedge(-dx_2-idx_3)\\
  &=f_{0^\prime0^\prime}\left(dx_0\wedge dx_2+dx_1\wedge dx_3+i (dx_0\wedge dx_3-dx_1\wedge dx_2)\right).
  \end{aligned}
\end{equation}
Then equation \eqref{eq:three section of LJ} implies that
\begin{equation*}
   f_{0^\prime0^\prime}d\bar{z}_0^{1^\prime}\wedge d\bar{z}_1^{1^\prime}\in C^\infty(U,\LJ).
\end{equation*}
Similarly, we can check that
\begin{equation}\label{eq:dx dz dengjia:2}
   f_{1^\prime1^\prime}d\bar{z}_1^{0^\prime}\wedge d\bar{z}_0^{0^\prime}=f_{1^\prime1^\prime}\left(-dx_0\wedge dx_2-dx_1\wedge dx_3+i (dx_0\wedge dx_3-dx_1\wedge dx_2)\right)\in C^\infty(U,\LJ).
\end{equation}
 And we have
\begin{equation}\label{eq:dx dz dengjia:3}
\begin{aligned}
   &\quad f_{0^\prime1^\prime}(-d\bar{z}_0^{0^\prime}\wedge d\bar{z}_1^{1^\prime}+d\bar{z}_1^{0^\prime}\wedge d\bar{z}_0^{1^\prime})\\
   &=f_{0^\prime1^\prime}\left(-(-dx_2+idx_3)\wedge(-dx_2-idx_3)+(dx_0+idx_1)\wedge(-dx_0+idx_1)\right)\\
   &=2if_{0^\prime1^\prime}(dx_0\wedge dx_1-dx_2\wedge dx_3)\in C^\infty(U,\LJ).
   \end{aligned}
\end{equation}
This implies that $\eta_U(f)\in C^\infty(U,\LJ)$.

\bigskip

By the definition of a resolution of a sheaf \cite{Wells08}, we only need to show that the sequence \eqref{eq:another resolution} is exact. We will divide the proof into the following steps.

\bigskip
 {\bf Step 1:} For any $x\in \Omega$, there is a neighborhood $U$ of $x$ such that $\eta_U$ is injective.
\bigskip

Assume that there exists $f_1,f_2\in   C^\infty(U,\odot^{2}\mathbb{C}^2)$ such that
\begin{equation*}
   f_1\neq f_2,\qquad \eta_U(f_1)=\eta_U (f_2).
\end{equation*}
It follows from the definition of $\eta_U$ \eqref{eq:def of eta} that
 $$(f_1)_{0^\prime0^\prime}=(f_2)_{0^\prime0^\prime},\quad (f_1)_{0^\prime1^\prime}=(f_2)_{0^\prime1^\prime},\quad (f_1)_{1^\prime1^\prime}=(f_2)_{1^\prime1^\prime}.$$
Recall that $\odot^2\mathbb{C}^2$ is the subspace of $\otimes^2\mathbb{C}^2$ satisfying the relations:
$$f_{0^\prime1^\prime}=f_{1^\prime0^\prime}.$$
This implies that
\begin{equation*}
   (f_1)_{0'0'} = (f_2)_{0'0'},\quad (f_1)_{0'1'} = (f_2)_{0'1'},\quad (f_1)_{1'1'} = (f_2)_{1'1'},\quad (f_1)_{1'0'} = (f_2)_{1'0'},
\end{equation*}
and hence \(f_1 = f_2\) as functions taking values in \(\otimes^2\mathbb{C}^2\), which leads to a contradiction.

\bigskip
 {\bf Step 2:} For any $x\in \Omega$, there is a neighborhood $U$ of $x$ such that
 \begin{equation*}
    \mathop{Ker}d\cap C^\infty(U,\LJ)=\mathop{Im}\eta_U.
 \end{equation*}

It follows from \eqref{eq:dx dz dengjia:1} \eqref{eq:dx dz dengjia:2} and \eqref{eq:dx dz dengjia:3} that
differential forms
$$d\bar{z}_0^{1^\prime}\wedge d\bar{z}_1^{1^\prime},\qquad -d\bar{z}_0^{0^\prime}\wedge d\bar{z}_1^{1^\prime}+d\bar{z}_1^{0^\prime}\wedge d\bar{z}_0^{1^\prime},\qquad d\bar{z}_1^{0^\prime}\wedge d\bar{z}_0^{0^\prime}$$
form a basis at each fiber space of $\LJ$. Then we can denote an element $f\in  C^\infty(U,\LJ)$ by
\begin{equation*}
   f=f_0d\bar{z}_0^{1^\prime}\wedge d\bar{z}_1^{1^\prime}+f_1(-d\bar{z}_0^{0^\prime}\wedge d\bar{z}_1^{1^\prime}+d\bar{z}_1^{0^\prime}\wedge d\bar{z}_0^{1^\prime})-f_2d\bar{z}_1^{0^\prime}\wedge d\bar{z}_0^{0^\prime}.
\end{equation*}

Note that the action of exterior differential operator $d$ on a smooth function $g$ can be represented as follows:
\begin{align*}
-2dg&= \left(\left(-\frac{\partial }{\partial x_2}-i\frac{\partial}{\partial x_3}\right)g\right)\, (dx_2-dx_3i)+\left(\left(-\frac{\partial}{\partial x_2}+i\frac{\partial}{\partial x_3}\right)g\right) \, (dx_2+dx_3i)\\
&\quad+\left(\left(-\frac{\partial}{\partial x_0}-i\frac{\partial}{\partial x_1}\right)g\right) \, (dx_0-dx_1i)+\left(\left(\frac{\partial}{\partial x_0}-i\frac{\partial}{\partial x_1}\right)g\right) \, (-dx_0-dx_1i)
   \\ &=\nabla_0^{0^\prime}gd\bar{z}_0^{0^\prime}+\nabla_0^{1^\prime}gd\bar{z}_0^{1^\prime}+\nabla_1^{0^\prime}gd\bar{z}_1^{0^\prime}+\nabla_1^{1^\prime}gd\bar{z}_1^{1^\prime}.
\end{align*}
Assume that $f\in \mathop{Ker}d\cap C^\infty(U,\LJ)$. One checks directly that
\begin{equation*}
\begin{aligned}
   -2df&=\nabla_0^{0^\prime}f_0d\bar{z}_0^{0^\prime}\wedge d\bar{z}_0^{1^\prime}\wedge d\bar{z}_1^{1^\prime}+\nabla_1^{0^\prime}f_0d\bar{z}_1^{0^\prime}\wedge d\bar{z}_0^{1^\prime}\wedge d\bar{z}_1^{1^\prime}\\
   &\quad -\nabla_0^{1^\prime}f_1d\bar{z}_0^{1^\prime}\wedge d\bar{z}_0^{0^\prime}\wedge d\bar{z}_1^{1^\prime}-\nabla_1^{0^\prime}f_1d\bar{z}_1^{0^\prime}\wedge d\bar{z}_0^{0^\prime}\wedge d\bar{z}_1^{1^\prime}\\
   &\quad +\nabla_0^{0^\prime} f_1d\bar{z}_0^{0^\prime}\wedge d\bar{z}_1^{0^\prime}\wedge d\bar{z}_0^{1^\prime}+\nabla_1^{1^\prime} f_1d\bar{z}_1^{1^\prime}\wedge d\bar{z}_1^{0^\prime}\wedge d\bar{z}_0^{1^\prime}\\
   &\quad -\nabla_0^{1^\prime}f_2d\bar{z}_0^{1^\prime}\wedge d\bar{z}_1^{0^\prime}\wedge d\bar{z}_0^{0^\prime}-\nabla_1^{1^\prime}f_2d\bar{z}_1^{1^\prime}\wedge d\bar{z}_1^{0^\prime}\wedge d\bar{z}_0^{0^\prime}.
\end{aligned}
\end{equation*}
By exchanging the order of terms in the above equation, we have
\begin{equation*}
     \begin{aligned}
   -2df&=\nabla_0^{0^\prime}f_0d\bar{z}_0^{0^\prime}\wedge d\bar{z}_0^{1^\prime}\wedge d\bar{z}_1^{1^\prime}-\nabla_0^{1^\prime}f_1d\bar{z}_0^{1^\prime}\wedge d\bar{z}_0^{0^\prime}\wedge d\bar{z}_1^{1^\prime}\\
   &\quad +\nabla_1^{0^\prime}f_0d\bar{z}_1^{0^\prime}\wedge d\bar{z}_0^{1^\prime}\wedge d\bar{z}_1^{1^\prime}+\nabla_1^{1^\prime} f_1d\bar{z}_1^{1^\prime}\wedge d\bar{z}_1^{0^\prime}\wedge d\bar{z}_0^{1^\prime}\\
   &\quad +\nabla_0^{0^\prime} f_1d\bar{z}_0^{0^\prime}\wedge d\bar{z}_1^{0^\prime}\wedge d\bar{z}_0^{1^\prime}-\nabla_0^{1^\prime}f_2d\bar{z}_0^{1^\prime}\wedge d\bar{z}_1^{0^\prime}\wedge d\bar{z}_0^{0^\prime}\\
   &\quad -\nabla_1^{0^\prime}f_1d\bar{z}_1^{0^\prime}\wedge d\bar{z}_0^{0^\prime}\wedge d\bar{z}_1^{1^\prime}-\nabla_1^{1^\prime}f_2d\bar{z}_1^{1^\prime}\wedge d\bar{z}_1^{0^\prime}\wedge d\bar{z}_0^{0^\prime}.
\end{aligned}
\end{equation*}
 This implies that
 \begin{equation*}
 \begin{aligned}
     0=-2df&=(\nabla_0^{0^\prime}f_0+\nabla_0^{1^\prime}f_1)d\bar{z}_0^{0^\prime}\wedge d\bar{z}_0^{1^\prime}\wedge d\bar{z}_1^{1^\prime}\\
     &\quad +(\nabla_1^{0^\prime}f_0+\nabla_1^{1^\prime}f_1)d\bar{z}_1^{0^\prime}\wedge d\bar{z}_0^{1^\prime}\wedge d\bar{z}_1^{1^\prime}\\
     &\quad +(\nabla_0^{0^\prime}f_1+\nabla_0^{1^\prime}f_2)d\bar{z}_0^{0^\prime}\wedge d\bar{z}_1^{0^\prime}\wedge d\bar{z}_0^{1^\prime}\\
     &\quad -(\nabla_1^{0^\prime}f_1+\nabla_1^{1^\prime}f_2)d\bar{z}_1^{0^\prime}\wedge d\bar{z}_0^{0^\prime}\wedge d\bar{z}_1^{1^\prime}.
     \end{aligned}
 \end{equation*}

 Note that the differential form
\[
d\bar{z}_0^{0^\prime}\wedge d\bar{z}_1^{1^\prime}\wedge d\bar{z}_0^{1^\prime}\wedge d\bar{z}_1^{0^\prime}=4dx_2\wedge dx_3\wedge dx_0\wedge dx_1
\]
is not identically zero everywhere. Assume there exists a set of constants
\[
\bigl\{C_A^{A'}\bigr\}_{A=0,1,\;A'=0',1'}
\]
for which the following identity holds:
\begin{equation*}
\begin{aligned}
0&=C_1^{0'} \, d\bar{z}_0^{0'} \wedge d\bar{z}_0^{1'} \wedge d\bar{z}_1^{1'}
+ C_0^{0'} \, d\bar{z}_1^{0'} \wedge d\bar{z}_0^{1'} \wedge d\bar{z}_1^{1'} \\
&\quad + C_1^{1'} \, d\bar{z}_0^{0'} \wedge d\bar{z}_1^{0'} \wedge d\bar{z}_0^{1'}
+ C_0^{1'} \, d\bar{z}_1^{0'} \wedge d\bar{z}_0^{0'} \wedge d\bar{z}_1^{1'}.
\end{aligned}
\end{equation*}
By wedging the differential forms \(d\bar{z}_A^{A^\prime}\) into the above equation, we obtain
\begin{align*}
  &C_1^{0'} d\bar{z}_1^{0'}\wedge d\bar{z}_0^{0'} \wedge d\bar{z}_0^{1'} \wedge d\bar{z}_1^{1'}=0,\quad C_0^{0'} d\bar{z}_0^{0'}\wedge d\bar{z}_1^{0'} \wedge d\bar{z}_0^{1'} \wedge d\bar{z}_1^{1'}=0,\\
  &C_1^{1'} d\bar{z}_1^{1'}\wedge d\bar{z}_0^{0'} \wedge d\bar{z}_1^{0'} \wedge d\bar{z}_0^{1'}=0,\quad C_0^{1'}d\bar{z}_0^{1'}\wedge d\bar{z}_1^{0'} \wedge d\bar{z}_0^{0'} \wedge d\bar{z}_1^{1'}=0.
\end{align*}
Hence each $C_A^{A^\prime}$ equals to $0$ and the differential forms
\[
d\bar{z}_0^{0^\prime}\wedge d\bar{z}_0^{1^\prime}\wedge d\bar{z}_1^{1^\prime},\quad
d\bar{z}_1^{0^\prime}\wedge d\bar{z}_0^{1^\prime}\wedge d\bar{z}_1^{1^\prime},\quad
d\bar{z}_0^{0^\prime}\wedge d\bar{z}_1^{0^\prime}\wedge d\bar{z}_0^{1^\prime},\quad
d\bar{z}_1^{0^\prime}\wedge d\bar{z}_0^{0^\prime}\wedge d\bar{z}_1^{1^\prime}
\]
are linearly independent everywhere.

Then equations
\begin{equation}\label{eq:prop:step 2:1}
\nabla_A^{0^\prime}f_0+\nabla_A^{1^\prime}f_1=0,\qquad \nabla_A^{0^\prime}f_1+\nabla_A^{1^\prime}f_2=0,
\end{equation}
hold for any $A=0,1$.

Now we need to find $\tilde{f}\in C^\infty(U,\odot^2\mathbb{C}^2)\cap\mathop{Ker}\mathscr{D}^{(2)}$ such that $\eta_U(\tilde{f})=f$. We take $\tilde{f}$ as follows:
\begin{equation*}
    \tilde{f}_{0^\prime0^\prime}=f_0,\qquad \tilde{f}_{0^\prime1^\prime}=\tilde{f}_{1^\prime0^\prime}=f_1,\qquad \tilde{f}_{1^\prime1^\prime}=f_2.
\end{equation*}
 It is clear that $\eta_U(\tilde{f})=f$. The fact that $\tilde{f}\in\mathop{Ker}\mathscr{D}^{(2)}$ follows from equation \eqref{eq:prop:step 2:1}.

\bigskip
 {\bf Step 3:} For any $x\in \Omega$, there is a neighborhood $U$ of $x$ such that for any complex valued closed $3$-form $g$ on $U$, there exists $f\in C^\infty(U,\LJ)$ such that $df=g$.

\bigskip

By taking $U$ to be a convex neighborhood of $x$ in Lemma \ref{lem:essential}, the proof of Step 3 follows directly from Lemma \ref{lem:essential}.

\bigskip
 {\bf Step 4:} For any $x\in \Omega$, there is a neighborhood $U$ of $x$ such that for any complex valued $4$-form $g$ on $U$, there exists a complex valued $3$-form $f$ such that $df=g$.
\bigskip

The proof of {\bf Step 4} follows from the well known Poincar\'{e} Lemma.
Finally, since the sheaf $\mathcal{E}$ is fine, the resolution is acyclic and equation \eqref{eq:H_1} holds.
\qed

\section{Proof of Theorem \ref{thm:real valued harmonic function} }\label{section:Proof of Theorem real valued harmonic function}

 This section is devoted to the proof of Theorem \ref{thm:real valued harmonic function}, which begins with the following lemma.
\begin{lemma}\label{lem:useful}
    Let $\Omega$ be a domain of $\mathbb{R}^4$ and $f_0,f_1,f_2$ be three real valued functions on $\Omega$. Let $F$ be an anti-self-dual $2$-form defined by
\begin{equation*}
     F=f_1(dx_0\wedge dx_1-dx_2\wedge dx_3)+f_2(dx_0\wedge dx_2+dx_1\wedge dx_3)+f_3(dx_0\wedge dx_3-dx_1\wedge dx_2).
\end{equation*}
    Then we have
    \begin{equation}\label{eq:useful lemma:main}
         D(f_1\mathbf{i}+f_2\mathbf{j}+f_3\mathbf{k})=(dF)_{123}-(dF)_{023}\mathbf{i}+(dF)_{013}\mathbf{j}-(dF)_{012}\mathbf{k}.
    \end{equation}
    Here each $(dF)_{\alpha\beta\gamma}$ is the component of $dF$ with respect to the form $dx_\alpha\wedge dx_\beta \wedge dx_\gamma$.
    \end{lemma}
\begin{proof}
 It follows from direct computation that the left hand side of equation \eqref{eq:useful lemma:main} equals to
    \begin{equation}\label{eq:useful lemma:1}
\begin{aligned}
     D(f_1\mathbf{i}+f_2\mathbf{j}+f_3\mathbf{k})&=-\frac{\partial}{\partial x_1}f_1-\frac{\partial}{\partial x_2}f_2-\frac{\partial}{\partial x_3}f_3\\
     &\quad +\left(\frac{\partial}{\partial x_0}f_1-\frac{\partial}{\partial x_3}f_2+\frac{\partial}{\partial x_2}f_3\right)\mathbf{i}\\
     &\quad +\left(\frac{\partial}{\partial x_3}f_1+\frac{\partial}{\partial x_0}f_2-\frac{\partial}{\partial x_1}f_3\right)\mathbf{j}\\
     &\quad +\left(-\frac{\partial}{\partial x_2}f_1+\frac{\partial}{\partial x_1}f_2+ \frac{\partial}{\partial x_0}f_3\right)\mathbf{k}.
     \end{aligned}
\end{equation}
On the other hand, we have
\begin{equation}\label{eq:useful lemma:2}
\begin{aligned}
    dF&=  \left(-\frac{\partial}{\partial x_1}f_1-\frac{\partial}{\partial x_2}f_2-\frac{\partial}{\partial x_3}f_3\right)dx_1\wedge dx_2 \wedge dx_3     \\
    &\quad +\left(-\frac{\partial}{\partial x_0}f_1+\frac{\partial}{\partial x_3}f_2-\frac{\partial}{\partial x_2}f_3\right)dx_0\wedge dx_2 \wedge dx_3   \\
    &\quad +\left(\frac{\partial}{\partial x_3}f_1+\frac{\partial}{\partial x_0}f_2-\frac{\partial}{\partial x_1}f_3\right)dx_0\wedge dx_1\wedge dx_3\\
    &\quad +\left(\frac{\partial}{\partial x_2}f_1-\frac{\partial}{\partial x_1}f_2- \frac{\partial}{\partial x_0}f_3\right)dx_0\wedge dx_1\wedge dx_2.
    \end{aligned}
\end{equation}
It follows from \eqref{eq:useful lemma:1} and \eqref{eq:useful lemma:2} that the equation \eqref{eq:useful lemma:main} holds true. This completes the proof.
\end{proof}
\bigskip

\emph{Proof of Theorem \ref{thm:real valued harmonic function}.}
The proof of the equivalence (3) $\Leftrightarrow$ (4) closely parallels the complex case treated in Theorem \ref{thm:C andH}. Consider the short exact sequence of sheaves:
$$0 \to \widetilde{\mathcal{R}}\hookrightarrow \mathcal{R}\xrightarrow{\operatorname{Re}} \mathcal{H}\to 0,$$
where $\mathcal{R}$ denotes the sheaf of quaternionic regular functions over $\Omega$, and $\mathcal{H}$ the sheaf of real-valued harmonic functions over $\Omega$. Exactness at $\mathcal{H}$ follows from the fact that every real-valued harmonic function is locally the real part of a quaternionic regular function \cite[Theorem 4]{Sudbery79}. This induces the long exact sequence in cohomology:
$$ 0  \rightarrow  H^0(\Omega, \widetilde{\mathcal{R}})  \rightarrow    H^0(\Omega, {\mathcal{R}}) \xrightarrow{\operatorname{Re}} H^0(\Omega, {\mathcal{H}}) \rightarrow  H^1(\Omega, \widetilde{\mathcal{R}})  \rightarrow   H^1(\Omega, \mathcal{R}) \rightarrow \cdots .$$
Recall from Lemma \ref{lem:solution of Cauchy Fueter} that $ H^1(\Omega, \mathcal{R})=0$. Hence, the map $H^0(\Omega, \mathcal{R}) \xrightarrow{\operatorname{Re}} H^0(\Omega, \mathcal{H})$ is surjective if and only if $H^1(\Omega, \widetilde{\mathcal{R}}) = 0$, which is precisely the equivalence (3) $\Leftrightarrow$ (4).

The proof of $(1)\Rightarrow (2)$ follows from Proposition \ref{prop:a new resolution} and Lemma \ref{lem:essential}.

Next, we will prove  $(2)\Rightarrow (3)$. Let $h$ be a real valued harmonic function on $\Omega$. Then the existence of quaternionic regular function $f$ with $\operatorname{Re}f=h$ is equivalent to the fact that there exists solutions $f_1, f_2, f_3$ of equation
\begin{equation}\label{eq:proof of thm1:2tui1:1}
    D(f_1\mathbf{i}+f_2\mathbf{j}+f_3\mathbf{k})=-Dh.
\end{equation}
By Lemma \ref{lem:useful}, we need to find an anti-self-dual $2$-from $F$ such that
\begin{equation*}
    (dF)_{123}-(dF)_{023}\mathbf{i}+(dF)_{013}\mathbf{j}-(dF)_{012}\mathbf{k}=-Dh.
\end{equation*}
More precisely, $F$ is required to satisfy that
\begin{equation}\label{eq:proof of thm1:2tui1:2}
    (dF)_{123}=-\frac{\partial}{\partial x_0}h,\quad  (dF)_{023}= \frac{\partial}{\partial x_1}h,\quad  (dF)_{013}=-\frac{\partial}{\partial x_2}h,\quad  (dF)_{012}=\frac{\partial}{\partial x_3}h.,
\end{equation}
or equivalently,
\begin{equation*}
   dF=H=\frac{\partial }{\partial x_0} hdx_1\wedge dx_2 \wedge dx_3-\frac{\partial }{\partial x_1} hdx_0\wedge dx_2 \wedge dx_3+ \frac{\partial }{\partial x_2} hdx_0\wedge dx_1 \wedge dx_3- \frac{\partial }{\partial x_3} hdx_0\wedge dx_1 \wedge dx_2.
\end{equation*}
Note that
\begin{equation*}
    dH=\Delta hdx_0\wedge dx_1\wedge dx_2\wedge dx_3=0.
\end{equation*}
Then the condition $H^1(\Omega,\mathcal{R}^{(2)})=0$ implies that there exists $\tilde{F}\in C^\infty(\Omega,\LJ)$ such that $d{\tilde{F}}=H$. It is clear that $\tilde{F}$ satisfies the requirement in equation \eqref{eq:proof of thm1:2tui1:2}.
Denote by
\begin{equation*}
     \tilde{F}=\tilde{F}_1(dx_0\wedge dx_1-dx_2\wedge dx_3)+\tilde{F}_2(dx_0\wedge dx_2+dx_1\wedge dx_3)+\tilde{F}_3(dx_0\wedge dx_3-dx_1\wedge dx_2).
\end{equation*}
Then Lemma \ref{lem:useful} implies that function
\begin{equation*}
     h+\tilde{F}_1\mathbf{i}+\tilde{F}_2\mathbf{j}+\tilde{F}_3\mathbf{k}
\end{equation*}
is a quaternionic regular function with real part $h$.

Finally, we need to prove $(3)\Rightarrow(1)$.

Assume that $g$ is a smooth real-valued closed $3$-form. Our goal is to find a $2$-form $H$ such that $dH=g$.

Denote by $d^\ast$ the formal adjoint operator of $d$. It is clear that  $dd^\ast+d^\ast d=-\Delta$ and
\begin{equation*}
    (dd^\ast+d^\ast d)G=-\sum_{0\leqslant\alpha<\beta<\gamma\leqslant 3}\Delta G_{\alpha\beta\gamma}dx_\alpha\wedge dx_\beta \wedge dx_\gamma
\end{equation*}
holds for any 3-form $G$. Denote by  $G_{\alpha\beta\gamma}$ the solutions of equations
\begin{equation*}
    \Delta G_{\alpha\beta\gamma}=g_{\alpha\beta\gamma}, \qquad 0\leqslant \alpha<\beta<\gamma\leqslant3,
\end{equation*}
 respectively. The existence of such $G$ follows from the facts that Dirac equations are always solvable for smooth functions on domain $\Omega$ \cite{BDS82} . Denote by
 \begin{equation*}
     G=\sum_{0\leqslant\alpha<\beta<\gamma\leqslant 3}G_{\alpha\beta\gamma}dx_\alpha\wedge dx_\beta \wedge dx_\gamma.
 \end{equation*}
 Then we have
 \[-dd^\ast G-d^\ast dG=g.\]
We only need to show that there exists an $H_1$ such that $dH_1=d^\ast dG$.

Due to the commutativity of the Laplacian $\Delta$ with the partial derivatives $\dfrac{\partial}{\partial x_i},i=0,1,2,3$ and since $dg=0$, it follows that $(dG)_{0123}$ is a real valued harmonic function. Assuming from  $(3)$ that there exists real valued functions $F_1,F_2,F_3$ such that $(dG)_{0123}+F_1\mathbf{i}+F_2\mathbf{j}+F_3\mathbf{k}$ is a quaternionic regular on $\Omega$. It follows from Lemma \ref{lem:useful} that there exists an anti-self-dual $2$-form $\hat{F}$ such that
\begin{equation*}
    -D((dG)_{0123})=(d\hat{F})_{123}-(d\hat{F})_{023}\mathbf{i}+(d\hat{F})_{013}\mathbf{j}-(d\hat{F})_{012}\mathbf{k}.
\end{equation*}
It is equivalent to following equations:
\begin{equation}\label{eq:proof of thm1:final:1}
\begin{aligned}
    &\frac{\partial }{\partial x_0}(dG)_{0123}=-(d\hat{F})_{123},\qquad \frac{\partial }{\partial x_1}(dG)_{0123}=(d\hat{F})_{023},\\
    &\frac{\partial }{\partial x_2}(dG)_{0123}=-(d\hat{F})_{013}, \qquad \frac{\partial }{\partial x_3}(dG)_{0123}= (d\hat{F})_{012}.
    \end{aligned}
\end{equation}
Notice that
\begin{equation}\label{eq:proof of thm1:finalfinal}
\begin{aligned}
    d^\ast dG&=-\frac{\partial }{\partial x_0}(dG)_{0123}dx_1\wedge dx_2\wedge dx_3 +\frac{\partial }{\partial x_1}(dG)_{0123}dx_0\wedge dx_2\wedge dx_3\\
    &\quad -\frac{\partial }{\partial x_2}(dG)_{0123}dx_0\wedge dx_1\wedge dx_3 +\frac{\partial }{\partial x_3}(dG)_{0123}dx_0\wedge dx_1\wedge dx_2.
    \end{aligned}
\end{equation}
It follows from \eqref{eq:proof of thm1:final:1} and \eqref{eq:proof of thm1:finalfinal} that
\begin{equation*}
    d^\ast dG=d\hat{F}.
\end{equation*}
Then we have
\begin{equation*}
    g=-(d^\ast d+dd^\ast )G=d(-\hat{F}-d^\ast G).
\end{equation*}
This completes the proof. \qed

\section{Proof of Theorem \ref{thm:inverse problem}}\label{section:Proof of Theorem inverse problem}

We now present the proof of Theorem \ref{thm:inverse problem}, which characterizes the representability of harmonic functions by $1$-regular functions via the operators $L_0$ and $L_1$.

\bigskip

\emph{Proof of Theorem \ref{thm:inverse problem}.}
Recall from the definition of $L_0$ and $L_1$ \eqref{eq:def of L0L1} that
\begin{equation}\label{eq:thm inverse: obviously:1}
    2h=L_0\left(\frac{\overline{z_0^{0^\prime}}}{|q|^2}h,\frac{\overline{z_0^{1^\prime}}}{|q|^2}h\right)+L_1\left(\frac{\overline{z_1^{0^\prime}}}{|q|^2}h,\frac{\overline{z_1^{1^\prime}}}{|q|^2}h\right).
\end{equation}
Define functions $f_X=(f_{X0},f_{X1})$ and $g_X=(g_{X0},g_{X1})$ by
\begin{equation}\label{eq:thm inverse: obviously:2}
    (f_{X0},f_{X1})=\left(\frac{\overline{z_0^{0^\prime}}}{|q|^2}h,\frac{\overline{z_0^{1^\prime}}}{|q|^2}h\right)+L_1X,\qquad (g_{X0},g_{X1})=\left(\frac{\overline{z_1^{0^\prime}}}{|q|^2}h,\frac{\overline{z_1^{1^\prime}}}{|q|^2}h\right)-L_0X.
\end{equation}
Here $X\in C^\infty(\Omega,\odot^2\mathbb{C}^2)$. It follows from Proposition \ref{prop:L0L1} that $\frac{1}{2}f_X,\frac{1}{2}g_X$ satisfy the equation
\[h=L_0\left(\frac{f_X}{2}\right)+L_1\left(\frac{g_X}{2}\right).\]
 It then suffices to find $X$ such that $f_X,g_X$ are $1$-regular on $\Omega$. More precisely, the function $X$ is required to satisfy following equations:
\begin{equation}\label{eq:thm inverse:2}
    \mathscr{D}^{(1)}(L_1X)=-\mathscr{D}^{(1)}\left(\frac{\overline{z_0^{0^\prime}}}{|q|^2}h,\frac{\overline{z_0^{1^\prime}}}{|q|^2}h\right),\qquad \mathscr{D}^{(1)}(L_0X)=\mathscr{D}^{(1)}\left(\frac{\overline{z_1^{0^\prime}}}{|q|^2}h,\frac{\overline{z_1^{1^\prime}}}{|q|^2}h\right).
\end{equation}
By Proposition \ref{prop:L0L1}, the requirement  \eqref{eq:thm inverse:2} is equivalent to
\begin{equation}\label{eq:thm inverse:3}
    L_1\mathscr{D}^{(2)}X=-\mathscr{D}^{(1)}\left(\frac{\overline{z_0^{0^\prime}}}{|q|^2}h,\frac{\overline{z_0^{1^\prime}}}{|q|^2}h\right),\qquad L_0\mathscr{D}^{(2)}X=\mathscr{D}^{(1)}\left(\frac{\overline{z_1^{0^\prime}}}{|q|^2}h,\frac{\overline{z_1^{1^\prime}}}{|q|^2}h\right).
\end{equation}
We require the explicit form of equation \eqref{eq:thm inverse:3}. Recall from the definition of $\nabla_A^{B^\prime}$ \eqref{preliminary:nabla AA prime}, $z_j^{k^\prime}$ \eqref{eq:def:zAAprime} and Lemma \ref{lem:computation z} that
\begin{equation}\label{eq:thm inverse:4}
    \overline{z_0^{0^\prime}}=z_1^{1^\prime},\qquad \overline{z_0^{1^\prime}}=-z_1^{0^\prime},\qquad \nabla_A^{B^\prime}|q|^2=-2z_A^{B^\prime}, \qquad \nabla_A^{B^\prime}z_j^{k^\prime}=-2\epsilon^{jA}\epsilon^{k^\prime B^\prime}.
\end{equation}
It follows from \eqref{eq:thm inverse:4} that
\begin{equation*}
\begin{aligned}
    -\left(\mathscr{D}^{(1)}\left(\frac{\overline{z_0^{0^\prime}}}{|q|^2}h,\frac{\overline{z_0^{1^\prime}}}{|q|^2}h\right)\right)_A
    &=-\nabla_A^{0^\prime}\left(\frac{z_1^{1^\prime}h}{|q|^2}\right)+\nabla_A^{1^\prime}\left(\frac{z_1^{0^\prime}h}{|q|^2}\right),\\
    \left(\mathscr{D}^{(1)}\left(\frac{\overline{z_1^{0^\prime}}}{|q|^2}h,\frac{\overline{z_1^{1^\prime}}}{|q|^2}h\right)\right)_A
    &=-\nabla_A^{0^\prime}\left(\frac{z_0^{1^\prime}h}{|q|^2}\right)+\nabla_A^{1^\prime}\left(\frac{z_0^{0^\prime}h}{|q|^2}\right),\\
 \end{aligned}
\end{equation*}
where $A=0,1$.

Recall that $z_0^{0^\prime}z_1^{1^\prime}-z_0^{1^\prime}z_1^{0^\prime}=|q|^2$. For each $A=0,1$, the following identities hold:
\[-\nabla_A^{0^\prime} \frac{z_1^{1^\prime}}{|q|^2} +\nabla_A^{1^\prime} \frac{z_1^{0^\prime}}{|q|^2}=\frac{2\epsilon^{1A}\epsilon^{1^\prime0^\prime}}{|q|^2}-\frac{2z_A^{0^\prime}z_1^{1^\prime}}{|q|^4}-\frac{2\epsilon^{1A}\epsilon^{0^\prime1^\prime}}{|q|^2}+\frac{2z_A^{1^\prime}z_1^{0^\prime}}{|q|^4}=
\begin{cases}
\frac{2}{|q|^2}, & \text{ if } A=0,\\
0, &  \text{ if }A=1,
\end{cases}\]
and
\[-\nabla_A^{0^\prime} \frac{z_0^{1^\prime}}{|q|^2} +\nabla_A^{1^\prime} \frac{z_0^{0^\prime}}{|q|^2}=\frac{2\epsilon^{0A}\epsilon^{1^\prime0^\prime}}{|q|^2}-\frac{2z_A^{0^\prime}z_0^{1^\prime}}{|q|^4}-\frac{2\epsilon^{0A}\epsilon^{0^\prime1^\prime}}{|q|^2}+\frac{2z_A^{1^\prime}z_0^{0^\prime}}{|q|^4}
=\begin{cases}
0, & \text{ if } A=0,\\
-\frac{2}{|q|^2}, &  \text{ if }A=1.
\end{cases} \]
This implies that
\begin{equation}\label{eq:thm inverse:5}
\begin{aligned}
    -\left(\mathscr{D}^{(1)}\left(\frac{\overline{z_0^{0^\prime}}}{|q|^2}h,\frac{\overline{z_0^{1^\prime}}}{|q|^2}h\right)\right)_A
    &=-\frac{z_1^{1^\prime}}{|q|^2}\nabla_A^{0^\prime} h+\frac{z_1^{0^\prime}}{|q|^2}\nabla_A^{1^\prime} h+\delta_A^{0}\frac{2h}{|q|^2}=(L_1\mathscr{D}^{(2)}X)_A,\\
    \left(\mathscr{D}^{(1)}\left(\frac{\overline{z_1^{0^\prime}}}{|q|^2}h,\frac{\overline{z_1^{1^\prime}}}{|q|^2}h\right)\right)_A
    &=-\frac{z_0^{1^\prime}}{|q|^2}\nabla_A^{0^\prime} h+\frac{z_0^{0^\prime}}{|q|^2}\nabla_A^{1^\prime} h-\delta_A^{1}\frac{2h}{|q|^2}=(L_0\mathscr{D}^{(2)}X)_A,
 \end{aligned}
\end{equation}
where $A=0,1$.

We define $Y\in C^\infty(\Omega,\mathbb{C}^2\otimes \mathbb{C}^2)$ by
\begin{equation*}
    \begin{pmatrix}
        Y_{00^\prime} & Y_{10^\prime}  \\
        Y_{01^\prime} &  Y_{11^\prime}
    \end{pmatrix}
    =
    \frac{1}{|q|^2}
    \begin{pmatrix}
    z_0^{1^\prime} & z_1^{1^\prime}\\
    -z_0^{0^\prime} & -z_1^{0^\prime}
    \end{pmatrix}
    \begin{pmatrix}
    \left(\mathscr{D}^{(1)}\left(\frac{\overline{z_0^{0^\prime}}}{|q|^2}h,\frac{\overline{z_0^{1^\prime}}}{|q|^2}h\right)\right)_0 &\left(\mathscr{D}^{(1)}\left(\frac{\overline{z_0^{0^\prime}}}{|q|^2}h,\frac{\overline{z_0^{1^\prime}}}{|q|^2}h\right)\right)_1\\
    \left(\mathscr{D}^{(1)}\left(\frac{\overline{z_1^{0^\prime}}}{|q|^2}h,\frac{\overline{z_1^{1^\prime}}}{|q|^2}h\right)\right)_0 & \left(\mathscr{D}^{(1)}\left(\frac{\overline{z_1^{0^\prime}}}{|q|^2}h,\frac{\overline{z_1^{1^\prime}}}{|q|^2}h\right)\right)_1
    \end{pmatrix}.
\end{equation*}
It is clear that
\begin{align*}
\begin{pmatrix}
    (-L_1Y)_0 & (-L_1Y)_1\\
    (L_0Y)_0 & (L_0Y)_1
\end{pmatrix}
&=
\begin{pmatrix}
    -z_1^{0^\prime}Y_{00^\prime}-z_1^{1^\prime}Y_{01^\prime} &  -z_1^{0^\prime}Y_{10^\prime}-z_1^{1^\prime}Y_{11^\prime}\\
     z_0^{0^\prime}Y_{00^\prime}+z_0^{1^\prime}Y_{01^\prime} & z_0^{0^\prime}Y_{10^\prime}+z_0^{1^\prime}Y_{11^\prime}
\end{pmatrix}
=
\begin{pmatrix}
    -z_1^{0^\prime} & -z_1^{1^\prime}\\
    z_0^{0^\prime} & z_0^{1^\prime}
\end{pmatrix}
\begin{pmatrix}
    Y_{00^\prime} & Y_{10^\prime}  \\
    Y_{01^\prime} &  Y_{11^\prime}
\end{pmatrix}\\
&= \begin{pmatrix}
    \left(\mathscr{D}^{(1)}\left(\frac{\overline{z_0^{0^\prime}}}{|q|^2}h,\frac{\overline{z_0^{1^\prime}}}{|q|^2}h\right)\right)_0 &\left(\mathscr{D}^{(1)}\left(\frac{\overline{z_0^{0^\prime}}}{|q|^2}h,\frac{\overline{z_0^{1^\prime}}}{|q|^2}h\right)\right)_1\\
    \left(\mathscr{D}^{(1)}\left(\frac{\overline{z_1^{0^\prime}}}{|q|^2}h,\frac{\overline{z_1^{1^\prime}}}{|q|^2}h\right)\right)_0 & \left(\mathscr{D}^{(1)}\left(\frac{\overline{z_1^{0^\prime}}}{|q|^2}h,\frac{\overline{z_1^{1^\prime}}}{|q|^2}h\right)\right)_1
    \end{pmatrix}.
\end{align*}
In summary, we have
\begin{equation}\label{eq:thm inverse: property of Y}
    L_1Y=-\mathscr{D}^{(1)}\left(\frac{\overline{z_0^{0^\prime}}}{|q|^2}h,\frac{\overline{z_0^{1^\prime}}}{|q|^2}h\right),\qquad L_0Y=\mathscr{D}^{(1)}\left(\frac{\overline{z_1^{0^\prime}}}{|q|^2}h,\frac{\overline{z_1^{1^\prime}}}{|q|^2}h\right).
\end{equation}
It follows from direct computation that
\begin{align*}
\begin{pmatrix}
    Y_{00^\prime} & Y_{10^\prime}  \\
    Y_{01^\prime} &  Y_{11^\prime}
\end{pmatrix}&=\frac{1}{|q|^4}
\begin{pmatrix}
    z_0^{1^\prime} & z_1^{1^\prime}\\
    -z_0^{0^\prime} & -z_1^{0^\prime}
    \end{pmatrix}
\begin{pmatrix}
    z_1^{1^\prime}\nabla_0^{0^\prime}h-z_1^{0^\prime}\nabla_0^{1^\prime}h-2h & z_1^{1^\prime}\nabla_1^{0^\prime}h-z_1^{0^\prime}\nabla_1^{1^\prime}h\\
    -z_0^{1^\prime}\nabla_0^{0^\prime}h+z_0^{0^\prime}\nabla_0^{1^\prime}h & -z_0^{1^\prime}\nabla_1^{0^\prime}h+z_0^{0^\prime}\nabla_1^{1^\prime}h-2h
\end{pmatrix}\\
&=\frac{1}{|q|^4}\begin{pmatrix}
    |q|^2\nabla_0^{1^\prime}h-2z_0^{1^\prime}h & |q|^2\nabla_1^{1^\prime}h-2z_1^{1^\prime}h\\
    -|q|^2\nabla_0^{0^\prime}h+2z_0^{0^\prime}h & -|q|^2\nabla_1^{0^\prime}h+2z_1^{0^\prime}h
\end{pmatrix}
\end{align*}
In summary, the equation
\begin{equation}\label{eq:thm inverse:D2X}
    \left(Y\right)_{Aj^\prime}=\sum_{k =0,1}\epsilon^{k j}\left(\frac{1}{|q|^2}\nabla_{A}^{k^\prime}h-2\frac{z_{A}^{k^\prime}}{|q|^4}h\right)
\end{equation}
holds for each $A,j=0,1$.

If there exists $X$ such that $\mathscr{D}^{(2)}X=Y$, then $X$ satisfies equation \eqref{eq:thm inverse:2} and $f_X,g_X$ are $1$-regular.

However, the solvability of $\mathscr{D}^{(2)}X=Y$ requires that $\mathscr{D}_1^{(2)} Y=0$. Thus, our next step is to check this condition.

It follows from the definition of $\mathscr{D}_1^{(2)}$ \eqref{eq:def of D1}, the expression for $Y$ in \eqref{eq:thm inverse:D2X} and equation \eqref{eq:thm inverse:4} that
\begin{equation}\label{eq:thm inverse:check D1Y equals to 0}
    \begin{aligned}
        \mathscr{D}_1^{(2)} Y&=\sum_{A,B,j=0,1}\epsilon^{BA}\nabla^{j^\prime}_AY_{Bj^\prime}\\
        &=\sum_{A,B,j,k=0,1}\epsilon^{BA}\epsilon^{kj}\left(\frac{1}{|q|^2}\nabla_{A}^{j^\prime}\nabla_B^{k^\prime}h+\nabla_{A}^{j^\prime}\frac{1}{|q|^2}\nabla_B^{k^\prime}h-2\frac{z_B^{k^\prime}}{|q|^4}\nabla_A^{j^\prime}h-2h\nabla_A^{j^\prime}\frac{z_B^{k^\prime}}{|q|^4}\right).
    \end{aligned}
\end{equation}
We divide the calculations of equation \eqref{eq:thm inverse:check D1Y equals to 0} into following parts.
The calculations of sum of the second derivative terms of $h$ are given as following
\begin{equation}\label{eq:thm inverse:check D1Y equals to 0:1}
    \sum_{A,B,j,k=0,1}\epsilon^{BA}\epsilon^{kj}\frac{1}{|q|^2}\nabla_{A}^{j^\prime}\nabla_B^{k^\prime}h=\frac{2}{|q|^2}(\nabla_0^{0^\prime}\nabla_{1}^{1^\prime}-\nabla_0^{1^\prime}\nabla_{1}^{0^\prime})h=\frac{2}{|q|^2}\Delta h=0.
\end{equation}
The calculations of sum of the first derivative terms of $h$ follows from
\begin{equation}\label{eq:thm inverse:check D1Y equals to 0:2}
\begin{aligned}
&\quad \sum_{A,B,j,k=0,1}\epsilon^{BA}\epsilon^{kj}\left(\nabla_{A}^{j^\prime}\frac{1}{|q|^2}\nabla_B^{k^\prime}h-2\frac{z_B^{k^\prime}}{|q|^4}\nabla_A^{j^\prime}h\right)\\
&=\sum_{A,B,j,k=0,1}\epsilon^{BA}\epsilon^{kj}\left(\nabla_{A}^{j^\prime}\frac{1}{|q|^2}\nabla_B^{k^\prime}h-\nabla_{B}^{k^\prime}\frac{1}{|q|^2}\nabla_A^{j^\prime}h\right)=0.
\end{aligned}
\end{equation}
Here we use the facts that $\epsilon^{BA}\epsilon^{kj}=\epsilon^{AB}\epsilon^{jk}$ and the sum is independent of the choice of summation indices. Finally, we calculate the sum of terms of $h$. It follows from equation \eqref{eq:thm inverse:4} that
\begin{equation}\label{eq:thm inverse:check D1Y equals to 0:3}
\begin{aligned}
     \sum_{A,B,j,k=0,1}2\epsilon^{BA}\epsilon^{kj}h\nabla_A^{j^\prime}\frac{z_B^{k^\prime}}{|q|^4}&=h\sum_{A,B,j,k=0,1}\left(-4\epsilon^{BA}\epsilon^{kj}\epsilon^{BA}\epsilon^{kj}\frac{1}{|q|^4}+8\epsilon^{BA}\epsilon^{kj}\frac{z_A^{j^\prime}z_B^{k^\prime}}{|q|^6}\right)\\
    &=-16h\frac{1}{|q|^4}+16h\frac{z_0^{0^\prime}z_1^{1^\prime}-z_0^{1^\prime}z_1^{0^\prime}}{|q|^6}=0.
    \end{aligned}
\end{equation}

It follows from \eqref{eq:thm inverse:check D1Y equals to 0}, \eqref{eq:thm inverse:check D1Y equals to 0:1}, \eqref{eq:thm inverse:check D1Y equals to 0:2} and \eqref{eq:thm inverse:check D1Y equals to 0:3} that
\begin{equation*}
    \mathscr{D}_1^{(2)}Y=0.
\end{equation*}
Then Theorem \ref{thm:real valued harmonic function} implies that there exists $X\in C^\infty(\Omega,\odot^2\mathbb{C}^2)$ such that
\begin{equation}\label{eq:thm inverse:solution of equation D1}
    \mathscr{D}^{(2)}X=Y.
\end{equation}
Combine \eqref{eq:thm inverse: obviously:1}, \eqref{eq:thm inverse: obviously:2}, \eqref{eq:thm inverse:3}, \eqref{eq:thm inverse: property of Y} and \eqref{eq:thm inverse:solution of equation D1} to get that
\begin{equation*}
    h=\frac{1}{2}L_0\left(\left(\frac{\overline{z_0^{0^\prime}}}{|q|^2}h,\frac{\overline{z_0^{1^\prime}}}{|q|^2}h\right)+L_1X\right)+\frac{1}{2}\left( \left(\frac{\overline{z_1^{0^\prime}}}{|q|^2}h,\frac{\overline{z_1^{1^\prime}}}{|q|^2}h\right)-L_0X\right)
\end{equation*}
and functions
\begin{equation*}
    \left(\frac{\overline{z_0^{0^\prime}}}{|q|^2}h,\frac{\overline{z_0^{1^\prime}}}{|q|^2}h\right)+L_1X,\qquad \left(\frac{\overline{z_1^{0^\prime}}}{|q|^2}h,\frac{\overline{z_1^{1^\prime}}}{|q|^2}h\right)-L_0X
\end{equation*}
are $1$-regular. This completes the proof. \qed

 \begin{remark}
The concept of $k$-regular domains was introduced in \cite{Wang25} and is defined as follows.
\begin{definition}\cite[P.15]{Wang25}
A domain $D$ is called a domain of $k$-regularity if ones can not find two nonempty open sets $D_1$ and $D_2$ such that (1) $D_1$ is connected, $D_1\nsubseteq D$ and $D_2\subseteq D_1\cap D$; (2) for each $k$-regular function $f$ on $D$, there is a $k$-regular function $\tilde{f}$ on $D_1$ satisfying $f=\tilde{f}$on $D_2$.
\end{definition}
In \cite{Wang25}, Wang also proved that every quaternionic linearly convex domain is a domain of $k$-regularity, which implies that any domain in four-dimensional space is a domain of $k$-regularity. Using Theorem \ref{thm:inverse problem} from our manuscript, we can provide an alternative proof for the case $k=1$  under the additional condition that \(H^3(\Omega, \mathbb{R}) = 0\).

Assume that \(\Omega\subset \mathbb{H}\) satisfies \(H^3(\Omega,\mathbb{R})=0\) and \(0\in \partial\Omega\). Then, by Theorem \ref{thm:inverse problem}, there exist two \(1\)-regular functions \(f\) and \(g\) such that \(L_0 f + L_1 g = 1\). Applying the Cauchy inequality, we obtain
\[
(|f|^2+|g|^2)(q) > \frac{C}{|q|^2},
\]
for some positive constant \(C\) and for all \(q\in \Omega\). As a result, on any open set of the form \(\Omega\cap B(0,\epsilon)\), at least one of \(|f|\) or \(|g|\) must be unbounded. If, for instance, \(|f|\) is unbounded, then there exists no \(1\)-regular function \(F\) on \(B(0,\epsilon)\cup \Omega\) such that \(F|_{\Omega}=f\).
\end{remark}

\bigskip

\textbf{Acknowledgements}

The authors would like to thank the referees for many valuable suggestions.

\end{document}